\documentclass[11pt]{article}
\usepackage[text={6.75in,9.5in},centering,letterpaper]{geometry}
\usepackage{amsfonts}
\usepackage{amsmath}
\usepackage{amsthm}
\usepackage{amssymb}
\usepackage{amscd}
\usepackage{enumerate}
\usepackage{hyperref}
\usepackage{showlabels}
\usepackage{algorithm}
\usepackage{algorithmic}
\usepackage{graphicx}

\newcommand{\beq}{\begin{equation}}

\newcommand{\pd}[2]{\dfrac{\partial #1}{\partial #2}}

\newcommand{\eeq}{\end{equation}}
\newcommand{\ut}{\tilde{u}}

\def\k0{\kappa_0}

\newtheorem{thm}{Theorem}[section]

\newtheorem{lem}[thm]{Lemma}
\newtheorem{prop}[thm]{Proposition}

\theoremstyle{definition}
 
\theoremstyle{Remark}

\numberwithin{equation}{section}

\title{Non-Autonomous Inertial Manifold Reduction\thanks{This work was supported in part by NSF grants \# DMS-1115408 and \# DMS-1419047.}}
\author{
Yu-Min Chung
\thanks{Department of Mathematics, College of William and Mary
{\tt (email: ychung@wm.edu)}.}
\and
Andrew Steyer
\thanks{Department of Mathematics, University of Kansas, Lawrence, KS 66045 USA
{\tt (email: asteyer@math.ku.edu)}.}
\and
Erik S. Van Vleck
\thanks{Department of Mathematics, University of Kansas, Lawrence, KS 66045 USA
{\tt (email: erikvv@ku.edu)}.}
}

\begin{document}

\maketitle

\begin{abstract}
Techniques are developed for decoupling dissipative differential equations. 
The approach considered is based upon obtaining a sufficient
gap in the time dependent linear portion of the equation that corresponds to the linear
variational equation. This is done using an orthogonal change of variables that has proven
useful in the computation of Lyapunov to decompose the differential equation in terms of
slow and fast variables. Numerically this is accomplished in our implementation using smooth,
time dependent Householder reflectors. The the nonlinear decoupling transformation or inertial
manifold is obtained by solving a boundary value problem (BVP) which allows for a Newton iteration
as opposed to the traditional Lyapunov-Perron approach via a fixed point iteration. Finally, the
efficacy of the technique is shown using some challenging examples.
\end{abstract}

\vskip -1.0in

\section{Introduction}


In this paper we develop numerical techniques for decoupling of dissipative nonlinear differential equations.
To address problems in which the given differential equation is not easily written in terms of
linear part that is decoupled, we form in a canonical way a time dependent linear part based upon
the coefficient matrix of the linear variational equation. Decoupling of differential equations to
obtain reduced dimensional problems that retain the essential dynamics is an important topic and
forms the basis for inertial manifold techniques. 

Our contribution in this paper is to develop decoupling transformations for nonlinear differential
equations. The motivation for our techniques comes from inertial manifold techniques, but with the
important difference that we form in a canonical way a time dependent linear part using the coefficient
matrix of the linear variational equation. This allows for decoupling of the linear (time dependent)
part using techniques that have proven useful for approximation of Lyapunov exponents.
In this paper we combine techniques that have proven useful in the calculation of 
Lyapunov exponents and Sacker-Sell spectrum with techniques that are a variation
of decoupling or inertial manifold techniques. 
We employ a boundary value problem (BVP) formulation to determine the graph that defines the inertial
manifold. This has several advantages over the classical Lyapunov-Perron technique. These include
solving the associated differential equations directly without the need for the computation of fundamental matrix
solutions and the nonlinear variation of constants formula. The method is easy to apply using standard BVP solvers
which allow for easy use of Newton iteration as opposed to fixed point iteration and allows for straightforward use of
continuation techniques.

Inertial manifolds, which are finite dimensional, exponentially attracting, 
positively invariant Lipschitz manifolds, play an important role in studying long time behaviors of dynamical systems.  
The dynamical system restricted on the inertial manifold yields an {\it inertial form reduction}, which shares all long-term dynamics of the original system.  Exploiting this theory can be of enormous practical importance.  Detailed simulations, stability and bifurcation calculations can be performed on the inertial form at a small fraction of the computational effort required to perform them on large-scale discretizations of the original systems.  Due to its importance, there are enormous amounts of literature about the theory of inertial manifolds.  The existence of inertial manifolds has been established for a variety of partial differential equations, such as the Kuramoto-Sivashinsky equation \cite{FNST88, CFNT89, dim_KSE_inertial_Wang}, Ginsburg-Landau equation \cite{NFT89}, Cahn-Hilliard equation \cite{NST89, NFT89}, many reaction diffusion equations \cite{MPS88, Jolly89}, and the sabra shell model of turbulence \cite{CLT06}.  The computation of inertial manifolds and approximate inertial manifolds has been carried out and studied by \cite{ChungJolly, JRT, center_manifold,AIFKSE, ComputingInertial2, ComputingInertial3, JauberteauRosierTemam90, FJKST88}. The theory of inertial manifolds has been generalized to non-autonomous dynamical system \cite{ComputingInertial4, CL97, GV97,koksch2002inertial}.

We take the approach here of decoupling the time dependent linear part of the equation
using techniques that have proven useful in the approximation of Lyapunov exponents.
We first employ an orthogonal change of variables $Q(t)$ that brings the time dependent
coefficient matrix for the linear part of the equation to upper triangular. Subsequently,
we will compute a change of variables that decouples the linear part . 
This then gives us equations of the form considered by Aulbach and Wanner 
in \cite{AulbachWanner}.
A similar change of 
variables has been employed to justify that Lyapunov exponents and Sacker-Sell spectrum
may be obtained from the diagonal of the upper triangular coefficient matrix 
(see section 5 of \cite{DVV02} and sections 4 and 5 of \cite{DVV07}).
The references \cite{DVVEncyc} and \cite{DJVV11} (see also the references therein) 
provide a summary and overview
of recent work on approximation of Lyapunov exponents and in obtaining the 
orthgonal change of variables $Q(t)$. In this paper we consider the smooth orthogonal change of variables
based upon continuous Householder reflectors as developed in \cite{DV1,DV2}.



The outline of this paper is follows. In section \ref{basicsetup} we outline the way in which we
reduce the original initial value differential equation via two time dependent change of variables
to a nonlinear, nonautonomous system with decoupled time dependent linear part. The basic algorithms
we will implement and further background justifying our technique are outlined in section \ref{basicalg}. These are iterative techniques that in
one case update the the time dependent linear part and decoupling change of variables and in the 
other case freeze these time dependent factors. In  section \ref{Imp} we provide details of our implementation, including a
review of time dependent Householder transformations and culminating with with an outline of the algorithm we have implemented.
In section \ref{dependenceonT} 
we investigate the dependence of the convergence on $T \approx \infty$, and how to determine $T$.
The applicability of our techniques is first shown for several standard test problems in 
section \ref{numericalresults} before we
turn our attention to several challenging nonlinear problems. We conclude with
a summary of the theoretical and numerical results we have obtained and with some directions for
future research. 

\section{Basic Setup}\label{basicsetup}

Consider the initial value differential equation
\beq\label{nonlinDE}
\dot u = f(u,t),\,\, u(t_0)=u_0
\eeq
where $f:U \times (t_0,\infty) \mathbb{R}^d$, $t > t_0$, and $U$ is some open subset of $\mathbb{R}^d$ with $u_0 \in \mathbb{U}$.  Denote the solution of \eqref{nonlinDE} by $u(t;u_0)$.  We rewrite \eqref{nonlinDE} as a nonautonomous linear inhomogenous equation by linearizing $f$ in space about $u(t;u_0)$ as $C(t) = f'(u(t;u_0),t)$ ( $'$ denotes the derivative of $f$ with respect to the $u$ variables) to obtain
\beq\label{lininh}
\dot u = C(t)u + (f(u,t)-C(t)u \equiv  C(t)u + N(u,t)
\eeq
There exists (see e.g. \cite{DV1,DV2}) an orthogonal time-dependent change of variables $u(t)=Q(t)z(t)$, where the orthogonal matrix function $Q(t)$ satisfies the differential equation
\[
\dot Q = Q S(Q, C(t)),\,\,\, Q(0)=Q_0,\,\,\, 
S(Q,C(t))_{ij} = \left\{
\begin{array}{rl}
(Q^T C(t)Q)_{ij},& i>j\cr
0, & i=j\cr
-(Q^T C(t)Q)_{ji},& i<j\cr
\end{array}
\right.
\]
so that $z(t)$ is such that $z(t_0) = Q(t_0)^T u(t_0)$ and also satisfies the differential equation
\beq\label{nonlinDEz}
\dot z = D(t)z + Q^T(t)N(Q(t)z,t), \quad D(t) = Q^T(t)C(t)Q(t)-Q^T(t)\dot{Q}(t)
\eeq
where $D(t)$ is of the form 
\beq\label{Ddef}
D(t) = \begin{pmatrix}A(t) & E(t) \cr 0 & B(t)\cr\end{pmatrix}.
\eeq
where $A(t) \in \mathbb{R}^{p\times p}$ is upper triangular, $E(t) \in \mathbb{R}^{p times d-p}$, and $B(t) \in \mathbb{R}^{d-p \times d-p}$.  We rewrite the system \eqref{nonlinDEz} as
\begin{equation}\label{decoup}
	\begin{cases}
		\dot{x} = A(t) x + F(t, x, y)\\
		\dot{y} = B(t) y + G(t, x, y).
	\end{cases}
\end{equation}
  By rewriting \eqref{nonlinDE} in the linear inhomogeneous form \eqref{lininh} and changing to the $z(t)$ variables, we have transformed to the system \label{decoup} where, at the linear level, the $y(t)$ variables are decoupled from the $x(t)$ variables.  In the following sections, by assuming that the system \label{decouple} satisfies a spectral gap condition, we use inertial manifold techniques to extend this decoupling at the linear level to a full decoupling at the nonlinear level and then explore various numerical methods for computing this nonlinear transformation and its applications.

%
%
%

\section{Basic Algorithms}\label{basicalg}

Consider the nonautonomous dynamical system
\begin{equation}
	\label{equ:original NDS}
	\begin{cases}
		\dot{x} = A(t) x + F(t, x, y),\quad x(t_0) = x_0,\\
		\dot{y} = B(t) y + G(t, x, y), \quad y(t_0) = y_0.
	\end{cases}
\end{equation}
In this setting, assume $A(t)$ has the strong stable Lyapunov exponents while 
$B(t)$ has the unstable, neutral, and weakly stable Lyapunov exponents where for 
$w = \begin{pmatrix} w_1\cr w_2\cr\end{pmatrix}$
in (\ref{decoup1}) and $C_{11}(t)$ and $C_{22}(t)$ in (\ref{Ddef}) we have 
$A(t)=C_{22}(t)$, $B(t)=C_{11}(t)$, $F(t,x,y)=H_2(t,w_1,w_2)$, and $G(t,x,y)=H_1(t,w_1,w_2)$.
Here $x$ and $y$ are elements of some Banach spaces $X$ and $Y$, respectively, and $A:\mathbb{R} \rightarrow \mathcal{L}(X)$, $B:\mathbb{R} \rightarrow \mathcal{L}(Y)$, $F:\mathbb{R}\times X \times Y \rightarrow X$, and $G:\mathbb{R}\times X \times Y \rightarrow Y$ are mappings satisfying the following assumptions.  We follow the framework in \cite{AulbachWanner}. 

\begin{enumerate}[(H1)]
	\item \label{H1} The mappings $A$ and $B$ are locally integrable and there exists $K\geq 1$ and $\alpha < \beta$ such that the evolution operators $\Lambda_A$ and $\Lambda_B$ of the homogeneous linear equations $\dot{x} = A(t)x$ and $\dot{y} = B(t)y$, respectively, satisfy the estimates
	\begin{align*}
		&\| \Lambda_A(t,s) \| \leq Ke^{\alpha (t-s)} \quad \text{ for all } t\geq s,\\
		&\| \Lambda_B(t,s) \| \leq Ke^{\beta (t-s)} \quad \text{ for all } t\leq s.
	\end{align*}
	\item \label{H2} $F(t,0,0) = 0$ and $G(t,0,0) = 0$ for all $t\in \mathbb{R}$. $H(t,\cdot):=(F(t, \cdot),G(t,\cdot))$ is a Lipschitz function
	\begin{equation*}
		\| H(t, z_1) - H(t, z_2) \| \leq L \| z_1 - z_2 \|.
	\end{equation*}
	\item {\it Spectral Gap condition}:
	\begin{equation}
		\label{equ:classic gap}
		L < \frac{\beta - \alpha}{2K}.
	\end{equation}
\end{enumerate}
Under (H1) - (H3), the result in \cite{AulbachWanner} shows that there exists $\Phi: \mathbb{R}\times X \rightarrow Y$ and $\Psi: \mathbb{R}\times Y \rightarrow X$ such that
\begin{equation}
	\label{equ:decoupled NDS}
	\begin{cases}
		\dot{x} = A(t) x + F(t, x, \Phi(t,x))\\
		\dot{y} = B(t) y + G(t, \Psi(t,y), y)
	\end{cases}
	.
\end{equation}


Recall from \cite{AulbachWanner} that
\begin{equation}
\label{eq:T map}
\mathcal{T}(\varphi, x_0, t_0)(t) = \Lambda_A(t,t_0)x_0 + \int_{t_0}^t \Lambda_A(t,s)F(\varphi(s))\;ds
-\int_t^{\infty} \Lambda_B(t,s)\;G(\varphi(s))\;ds, \quad \forall t \geq t_0,
\end{equation}
and it can be shown that for given $x_0\in X$ and $t_0\in \mathbb{R}$, $\mathcal{T}$ has a fixed point in a proper Banach space, denoted by $\varphi$:
\begin{equation}
\label{eq:T map fixed point}
\varphi(t) = \mathcal{T}(\varphi(t)).
\end{equation}
Moreover, from \eqref{eq:T map fixed point} one can show that $\varphi(t) =: (x(t),\; y(t))$ is the unique solution of
\begin{equation}
\label{equ:T map ivp}
\begin{cases}
\dot{x} = A(t)x + F(t,x,y),\quad x(t_0)= x_0, \\
\dot{y} = B(t)y + G(t,x,y), \quad y(t_0) = -\int_{t_0}^{\infty} \Lambda_B(0,s)\;G(\varphi(s))\;ds.
\end{cases}
\end{equation}
On the other hand, note from \eqref{eq:T map fixed point} that, as $t\rightarrow \infty$, $y(t) \rightarrow 0$.  Therefore, we can view \eqref{equ:T map ivp} as the following BVP
\begin{equation}\label{SMBVP}
\begin{cases}
\dot{x} = A(t)x + F(t,x, y), \quad x(t_0) = x_0 \\
\dot{y}	= B(t)y + G(t,x, y), \quad y(\infty) = 0
\end{cases}
.
\end{equation}
Similarly, for the inertial manifold, one can consider the following BVP
\begin{equation}\label{IMBVP}
\begin{cases}
\dot{x} = A(t)x + F(t,x, y), \quad x(-\infty) = 0 \\
\dot{y}	= B(t)y + G(t,x, y), \quad y(t_0) = y_0
\end{cases}
.
\end{equation}

Our approach here can be thought of as an expansion in Lyapunov vectors. We
write the original variable $u(t) = Q(t) v(t)$ where $v(t) = \begin{pmatrix}y(t) \cr x(t) \end{pmatrix}$ and 
$Q(t)$ is the orthogonal time dependent change of variables. This will tend to organize the growth and decay in
the linearized equation so that $y(t)$ corresponds to the positive Lyapunov exponents and some of the larger
negative Lyapunov exponents while $x(t)$ corresponds to the more negative Lyapunov exponents. For $y(t)\in\mathbb{R}^p$,
$Q_p(t)$ denote the first $p$ columns of $Q(t)$ and $Q_{n-p}(t)$ denote the remaining columns of $Q(t)$. We can 
form time dependent projections $P_1(t) = Q_p(t)Q_p^T(t)$ and $P_2(t) = Q_{n-p}(t)Q_{n-p}^T(t)$. It is easy to see that
these are complementary projections. Writing $u(t) = P_1(t) w_1(t) + P_2(t) w_2(t)$ we have immediately that 
$y(t) = Q_p^T(t) w_1(t)$ and $x(t) = Q_{n-p}(t) w_2(t)$. 


To justify this as expansion in terms of Lyapunov vectors we recall the approach taken in
\cite{DEVV10,DEVV11}. Let $X(t)$ be a fundamental matrix solution of the linear variational
equations $w' = f'(u(t)) w$ and write $X(t) = Q(t) R(t)$, the smooth $QR$ decomposition
of $X(t)$. If we let $D(t) = {\rm diag}(R(t))$ and write 
$
R(t) = D(t) [D^{-1}(t) R(t)] \equiv D(t)Z(t),
$
then $Z(t)$ (unit upper triangular) satisfies an equation of the form
$\dot Z = E(t) Z$. 
Moreover, integral separation implies $E(t)\to 0$ as $t\to\infty$, so $Z(t)\to {\overline Z}$ and
as was shown in \cite{DEVV10} the rate of convergence is exponential based upon the integral separation.
This was also extended to the case of robust but non-distinct Lyapunov exponents.
To obtain Lyapunov vectors with respect to the  principal matrix solution ($X(0)=I$), we may assume that
$Q(0)=Q_0$ and $R(0)=I$. We want initial conditions $x_j(0)$ that grows/decays at exponential rate $\lambda_j$ as
$\to\infty$ and since in the case of robust Lyapunov exponents, the exponents are obtained from the diagonal
of $R(t)$ we want ${\overline Z} x_j(0) = e_j$ where $e_j$ is the $j$th unit vector,
so $x_j(0) = Q_0{\overline Z}^{-1}e_j$ for $j=1,2...$.

At an arbitrary time $t_0$, these Lyapunov vectors have the form
$
Q(t_0) {\overline Z}_{t_0}^{-1} e_j,\,\, j=1,2,...
$
where (if the Lyapunov exponents are robust) ${\overline Z}_{t_0} = \lim_{t\to\infty} Z_{t_0}(t)$,
$
\dot Z_{t_0}(t) = E_{t_0}(t) Z_{t_0}(t),\,\, Z_{t_0}(t_0) = I,
$
where $E_{t_0}(t)$ is strictly upper triangular and for $i<j$,
$
(E_{t_0})_{ij}(t) = B_{ij}(t) \exp(\int_{t_0}^t B_{jj}(s) - B_{ii}(s)ds).
$
Thus, the span of the first $p$ Lyapunov exponents at an arbitrary time $t_0$ 
is the same as the span of the first $p$ columns of $Q(t_0)$.


The approach we will take to determine $C(t)$ in the form (\ref{Cdef}) is based upon the use of
continuous Householder reflectors as developed in \cite{DV1,DV2}. This minimizes the number of equations
needed to obtain the form (\ref{Cdef}) and will result in $C_{11}(t)$ upper triangular, but $C_{22}(t)$ 
potentially full. An alternative is to solve for the first $p$ columns of $Q(t)$ (analogous to a reduced 
QR decomposition) which is enough to determine
$C_{11}(t)$ but requires a smooth orthogonal complement in order to obtain $C_{12}(t)$ and $C_{22}(t)$.
We note that the Householder approach in the time dependent case also requires that the sign of the diagonal
elements of $R(t)$ be allowed to vary for numerical stability purposes, but it is easy to recover the smooth
$Q(t)$ corresponding to positive diagonal elements in $R(t)$ by multiplying $Q(t)$ of the right and $R(t)$
on the left by an appropriate diagonal matrix of $\pm 1$s (see \cite{DV2}).

For a time dependent linear problem with $C(t)$ given in (\ref{Cdef}), the boundary value problem (\ref{IMBVP})
has the form $A(t):=C_{22}(t)$, $F\equiv 0$, $B(t):=C_{11}(t)$, and $G(x,y,t):= C_{12}(t)x$.  
Under the spectral gap condition (\ref{equ:classic gap}),
$x(t)\equiv 0$ and $(0,y_0)$ in the graph of the inertial
manifold provided the solution to $\dot y = B(t) y,\,\,\, y(0)=y_0$ is bounded as $t\to\infty$, in particular,
if $y_0$ is a linear combination of the Lyapunov vectors of $\dot y = B(t) y$ corresponding to positive Lyapunov
exponents.
Since $y(t)$ corresponds to positive Lyapunov exponents and $x(t)$ corresponds to negative exponents and the graph of the inertial manifold is defined to be the collection of initial conditions such that their trajectories are bounded for all backward in time, $x(t)\equiv 0$ and $(0,y_0)$ in the graph of the inertial
manifold.

\section{Implementation}\label{Imp}

Let $X(t)$ be a fundamental matrix solution of $\dot{u}(t) = C(t)u(t)$ and let $X(t) = Q(t)R(t)$ be a $QR$ factorization of $X(t)$ where $Q(t)$ is orthogonal and expressed product of Householder matrices $Q^T(t) = Q_p(t)\cdot \hdots \cdot Q_1(t)$ and $R(t)$ is of the form 
$$  \begin{pmatrix}R_{1,1}(t) & R_{1,2}(t) \cr 0 & R_{2,2}(t)\cr\end{pmatrix}$$
where $R_{1,1}(t) \in\mathbb{R}^{p\times p}$ is upper triangular, $R_{1,2}(t) \in \mathbb{R}^{p\times (d-p)}$, and $R_{2,2}(t) \in \mathbb{R}^{(d-p)\times (d-p)}$.  Since $Q_i(t)$ is a Householder matrix we can write $Q_i = \left[\begin{array}{cc} I_{i-1} & 0 \\0 & P_i(t) \end{array} \right]$ where $I_{i-1}$ is the $i-1$ dimensional identity matrix and $P_i(t) = I -2 v_i(t) v_i(t)^T$ where $v_i(t) \in \mathbb{R}^{d-i+1}$ with $\|v_i\|_2 =1$.  To figure out what $v_i$ needs to be, let $X_i$ be the transformed matrix $Q_{i-1}\cdot \hdots \cdot Q_1 X$ and then notice if $u_i = X_i e_i -\sigma_i \|X_i e_i\|_2 e_i$, then $v_i = v_i/\|v_i\|_2$ defines a Householder transformation $Q_i$ that diagonalizes the $i^{th}$ column of $X(t)$.  The value $\sigma_i$ is chosen to ensure numerical stability and the canonical choice is (see e.g. \cite{GVL96})
\begin{equation}\label{numstab}
\sigma_j = \left\{\begin{array}{cc} -1 & \text{ if } e_1^T x_i(t) \geq 0 \\
 1 & \text{ if } e_1^T x_i(t) < 0 \end{array} \right.
\end{equation}
where $x_i(t)$ is the $i^{th}$ column of $X_i(t)$.  To further reduce the number of equations we need to solve we can define $w_i = v_i/(e_1^T v_i)$ and then notice that we can recover $v_i$ from $w_i$ via the additional relationship $e_1^T v_i = -\sigma_i / \|w_i\|_2$.  We want to avoid computing the fundamental matrix solution $X(t)$ or any of its columns and want to work only with the coefficient matrix $C(t)=f'(u(t))$ and the $w_i$'s.  To obtain $D(t)$ from $C(t)=f'(u(t))$, let $D_i(t)$ be the matrix obtained after diagonalizing the $i^{th}$ column of $X$.  We find can inductively find $C_i$ by doing a sequence of $(C,Q_i)$ updates:
\begin{equation}\label{LQudpate}
C_i(t) = Q_i(t)C_{i-1}(t)-Q_i(t)\dot{Q}_i(t), \quad i=1,\hdots,i-1
\end{equation}
where $C_0$ is taken to be $f'(u(t))$.  We can express the $(C,Q_i)$ update in terms of only the $C_i$ and $w_i$ as
$$Q_i(t)C_{i-1}(t)-Q_i\dot{Q}_i = C_{i-1}-\frac{2}{w_i^T w_i}\left( w_i (w_i^T C_{i-1})+(C_{i-1} w_i) w_{i}^T \right) = 4\frac{w_i^T C_{i-1} w_{i}}{(w_i^T w_i)^2}w_i w_i^T - \frac{2}{w_i^T w_i}(w_i \dot{w}_i^T-\dot{w}_i w_i^T).$$
From the definition of $w_i$ we have $w_i = [1\,\,\, \hat{w}_i]^T$ and we can derive the following differential equation satisfied by $\hat{w}_i$:
\begin{equation}\label{weqn}
\dfrac{d \hat{w}_i}{dt} = \left(C_{i-1}^{1,1}]\hat{w}_i+\hat{w}_i^T \hat{C}^{1}_{i-1} - 2\frac{w_i^T C_{i-1} w_i}{w_i^T w_i}\right)\hat{w}_i +\left(1-\frac{w_i^T w_i}{2} \right) \hat{C}_{i-1}^{1} + \hat{C}_{i-1} \hat{w}_{i-1} \equiv h_i(\hat{w},t)
\end{equation}
where $C_{i-1}^{j,k}$ is the $(j,k)$ entry of $C_{i-1}$, $\hat{C}_{i-1}$ is the submatrix $C_{i-1}^{2:n,2:d}$ of $C_{i-1}$ and $\hat{C}^{1}_{i-1}$ is the column vector $C_{i-1}^{2:n,1}$.
The condition \eqref{numstab} is expressed by having 
\begin{equation}\label{wreembedcond}
1-\hat{w}_i^T \hat{w}_i \geq 0
\end{equation}
be satisfied at all times.  If \eqref{wreembedcond} is not satisfied, then we redefine $\sigma_i$ accordingly and then reembed the new update $\hat{w}_i$ variables to be consistent with the new $\sigma_i$.\\

We want to compute $Q(t)$ simultaneously with our computations of the solution $x(t)$ and $y(t)$.  To do so we must also include the differential equations for the $\hat{w}_i(t)$ variables in the boundary value problem.  We have the option of choosing a boundary condition at $-\infty$ or at $0$.  We choose the former since the transformation $Q(t)$ decouples the strong stable modes from the weakly stable and unstable modes in forward time and does the opposite in backward time.  Therefore we express the full boundary problem formulation to compute the inertial manifold at time $\tau$ as
\begin{equation}\label{bvpQ}
\left\{\begin{array}{c}
\dot{x} = A(t)x + F(t,x,y)\\
\dot{y} = B(t)y+G(t,x,y)\\
\hat{w}_i = h_i(\hat{w}_{i-1},t), \quad i=1,\hdots,p\\
x(-\infty) = 0, y(\tau) = y_0, \hat{w}_i(-\infty) = \hat{w}_{i,-\infty}
\end{array}
\right.
\end{equation}
In addition to \eqref{bvpQ} we have the initial value problem
\begin{equation}\label{Qivp}
\left\{\begin{array}{c}
\dot{x} = A(t)x + F(t,x,y)\\
\dot{y} = B(t)y+G(t,x,y)\\
\hat{w}_i = h_i(\hat{w}_{i-1},t), \quad i=1,\hdots,p\\
x(t_0) = x_0, y(t_0) = y_0, \hat{w}_i(t_0)=\hat{w}_{i,0}
\end{array}
\right.
\end{equation}
that is well defined for any $t_0 \in \mathbb{R}$.  The boundary conditions $\hat{w}_i(-\infty) = \hat{w}_{i,-\infty}$, $i=1,\hdots,p$ can be chosen at random since in forward time the columns of $Q(t)$ align at an exponential rate so that the system decouples into a strongly stable component and an unstable or weakly stable component.  However, different boundary conditions correspond to different solutions $x(t)$ and $y(t)$ since we have the relation $(x(t)^T,y(t)^T)^T = Q(t)^T u(t)$.  This is a drawback of our formulation since it requires us to make an arbitrary choice at the outset of our computation without a way of determining what the potential consequences are for our choice of $\hat{w}_i(-\infty) = \hat{w}_{i,-\infty}$.\\

Since the boundary value problem \eqref{bvpQ} boundary conditions at $-\infty$ it cannot be implemented in exact form.  We must choose a value $T > 0$ to use as our approximate $-\infty$ which leads to the following boundary value problem.
\begin{equation}\label{bvpQT}
\left\{\begin{array}{c}
\dot{x} = A(t)x + F(t,x,y)\\
\dot{y} = B(t)y+G(t,x,y)\\
\hat{w}_i = h_i(\hat{w}_{i-1},t), \quad i=1,\hdots,p\\
x(-T) = 0, y(t_0) = y_0, \hat{w}_i(-T) = \hat{w}_{i,-T}
\end{array}
\right.
\end{equation}
The value of $T$ will vary depending on the problem and the initial conditions.  It may not be obvious what value to use and for larger values the approximation may in fact get worse.  We will return to this issue in Section 5 where we discuss the convergence dependence of the approximation on $T$.   In section 6, we demonstrate that, contrary to what one would expect, the best value for $T$ may be quite small.  

To solve for the $\hat{w}_i$ equations using Matlab's IVP and BVP solvers we must modify the solvers so that we can reembed the $\hat{w}_i$ variables when the $\sigma_i$'s are changed according to \eqref{wreembedcond}.  To do so we modify the ode45 code so that we perform a reembedding on the value $\hat{w}_i(t_n)$ for $i=1,\hdots,p$ just before ode45 computes the candidate value of $\hat{w}_i(t_{n+1})$ for $i=1,\hdots,p$.  Similarly we modify the bvp4c code so that the reembedding happens after the Newton iteration converges to the candidate solution $\hat{w}_i(t_{n+1})$ for $i=1,\hdots,p$.\\

 We use the following algorithm to approximate the solution to \eqref{bvpQ} at time $t$.

  \begin{algorithm}\label{IMalg}
        \label{alg:general}
        \begin{algorithmic}[1]
                \REQUIRE $T,y_0$
                \ENSURE $w_i(t-T),y(t-T)$
                \STATE{(Construct Initial guess for BVP) Solve \eqref{Qivp} on $[t-T,t]$ using modified ode45 with the initial conditions $y(t-T)$, $x(t-T)=0$, and $w_{i}(t-T)$ for $i=1,\hdots,p$. }
                \STATE{(Solve BVP) Use the computed initial guess solution to solve \eqref{bvpQT} on $[t-T,t]$ with modified bvp4c }
                \end{algorithmic}
\end{algorithm}

We remark here that the choice of values for $\hat{w}_{i,-T}$ can have a large impact on the computation.  For each choice of  $p$, $-T$, and $\{\hat{w}_{i,-T }\}_{i=1}^{p}$ there is a unique $Q(t)$.  Each choice of $-T$ and $\{\hat{w}_{i,-T }\}_{i=1}^{p}$ determines a different decoupling of \eqref{nonlinDE} into unstable and weakly stable variables $x(t)$ and strongly stable variables $y(t)$.  This gives us a degree of flexibility in our computations; if our method fails to accurately approximate a point on the inertial manifold, we can increase the size of $p$, change the value of $T$ or select a different boundary condition $\{\hat{w}_{i,-T }\}_{i=1}^{p}$.

%
%
%

 equivalent to solving the initial value problem $\dot{y}=B(t)y+G(t,\psi(t,y),y)$, $y(t_0)=y_0$.  To solve this problem we must be able to evaluate $\psi(t,y(t))$ which requires the solution of a boundary value problem of the form \eqref{bvpQT}.  For example, using the explicit Euler scheme with fixed step-size to solve $\dot{y}=B(t)y+G(t,\psi(t,y),y)$, $y(t_0)=y_0$ would proceed as follows.  Fix a step-size $\Delta t> 0$ and let $t_n = n\Delta t$.  For $n \geq 0$ Let $y_n$ be the numerical solution at step $n$.  First we must compute the solution of the boundary value problem \eqref{bvpQT} with boundary conditions $x(t_n-T)=0$, $y(t_n)=y_n$, and $\hat{w}_i(t_n)=\hat{w}_{i,n}$ to obtain $x_n :=\psi(t_n,y_n)$.  Then we set $y_{n+1} = y_n +\Delta t \left(B(t_n)y_n+G(t_n,x_n,y_n) \right)$ and iterate the process.\\

  To use more sophisticated IVP solvers to solve $\dot{y}=B(t)y+G(t,\psi(t,y),y)$, $y(t_0)=y_0$ may require more calls to the boundary value problem solver \ref{IMalg}.  For instance to use an explicit $s$-stage Runge-Kutta method we will need to call the boundary value problem solver $s$ times to form the stage values.  It is interesting to note that if we solve the initial value problem $\dot{y}=B(t)y+G(t,\psi(t,y),y)$, $y(0)=y_0$ with a $k$-step linear multistep method of the form
  $$\sum_{i=0}^{k} \alpha_i y_{n+i} = \sum_{i=0}^{k} \beta_i (B(t_{n+i})y_{n+i}+G(t_{n+i},\psi(t_{n+i},y_{n+i}),y_{n+i}))$$
  then we only need to call the boundary value problem solver once at each time-step to form $(B(t_{n+k})y_{n+k}+G(t_{n+k},\psi(t_{n+k},y_{n+k}),y_{n+k}))$ since the other values will be stored in memory.  Using implicit methods in this context is inefficient, especially for nonlinear problems, since these may require many calls to \ref{IMalg} during an iterative process that solves a system of algebraic equations to compute a single step of the approximation.  \\
  
  With this in mind let  $y_{n+1} = \varphi(\Delta t,t_n,\{y_n\}_{k=0}^{n})$ be an $s$-stage, $k$-step explicit numerical method that approximates the initial value problem  $\dot{y}=B(t)y+G(t,\psi(t,y),y)$, $y(t_0)=y_0$.  We use the following algorithm to compute an approximate trajectory on the inertial manifold.

  \begin{algorithm}\label{IMalgts}
  
        \label{alg:generalts}
        \begin{algorithmic}[1]
                \REQUIRE $T,\Delta t$, $N$
                \ENSURE $\{y_n\}_{n=0}^{N}$
                 \FOR{$n=0:N-1$} 
                 \STATE{Set $t_n = t_0+n\Delta t$}
                 \STATE{(Construct Initial guess for BVP) Solve \eqref{Qivp} on $[-T,t_n]$ using modified ode45 with the initial conditions $y(t_n-T)$, $x(t_n-T)=0$, and $w_{i}(t_n-T)\}$ for $i=1,\hdots,p$.}
                \STATE{(Solve BVP) Solve \eqref{bvpQT} on $[t_n-T,t_n]$ with the boundary conditions $x(t_n-T)=0$, $y(t_n-T)=y_n$, and $\hat{w}_i(t_n-T)=\hat{w}_{i,-T}$ using modified bvp4c}
                \STATE{Set $y_{n+1} = \varphi(\Delta t,t_n,\{y_n\}_{k=0}^{n})$ (Repeating the above step to approximate stage values needed to evaluate the method)}
                \ENDFOR
                
        \end{algorithmic}
\end{algorithm}.\\

In Section 6 we present the results of some numerical experiments we conduct with Algorithms \ref{IMalg} and \ref{IMalgts} on several challenging problems.


\section{Convergence Dependence on $T$}\label{dependenceonT}

Throughout this article, we deal with the asymptotic boundary value problem (see \cite{Baxley1990122} and \cite{LentiniKeller80}).  The approach we take is to consider the boundary condition at the finite time, $T$.  More precisely, consider
\begin{equation}
\label{equ:stable bvp}
\begin{cases}
\dot{x} = A(t)x + F(t,x, y), \quad x(t_0) = x_0 \\
\dot{y}	= B(t)y + G(t,x, y), \quad y(\infty) = 0
\end{cases}
,
\end{equation}
and we approximate \eqref{equ:stable bvp} by
\begin{equation}
\label{equ:stable bvp T}
\begin{cases}
\dot{x} = A(t)x + F(t,x, y), \quad x(t_0) = x_0 \\
\dot{y}	= B(t)y + G(t,x, y), \quad y(T) = 0
\end{cases}
.
\end{equation}
In this section, we investigate the convergence dependence on $T$.
A natural question arises---does the solution of \eqref{equ:stable bvp T} converges to the solution of \eqref{equ:stable bvp} as $T$ goes to infinity?  This problem is also known as the asymptotic boundary value problem (see \cite{Baxley1990122} and \cite{LentiniKeller80}).  We recall the mapping $\mathcal{T}:\mathcal{F}_{\sigma}\times X \times \mathbb{R} \rightarrow \mathcal{F}_{\sigma}$ defined by \begin{equation}
\mathcal{T}(\varphi, x, t_0)(t) = \Lambda_A(t,t_0) x + \int_{t_0}^t \Lambda_A(t,s)\;F(\varphi(s))\;ds - \int_t^{\infty}\Lambda_B(t,s) \;G(\varphi(s))\;ds,\quad t\geq t_0,
\end{equation} where \begin{equation}\label{def Fsigma} \mathcal{F}_{\sigma,t_0}=\{\varphi \in C([t_0,\infty),Z); \; \|\varphi\|_{\sigma}=\sup_{t\geq t_0}e^{-\sigma (t - t_0)}\|\varphi(t)\|<\infty \}, \;\sigma\in (\alpha+KL,\beta-KL). \end{equation}
Let $\varphi$ be the fixed point of $\mathcal{T}$, and thus, it satisfies \eqref{equ:stable bvp}.  Now, we define the analogous $\mathcal{T}$ and $\mathcal{F}_{\sigma, t_0}$ for \eqref{equ:stable bvp T}.  Let   
\begin{equation}
\label{def:Fsigma T} \mathcal{F}_{\sigma, t_0,T}=\{\varphi \in C([t_0,T],Z); \; \|\varphi\|_{\sigma, T}:=\max_{t_0\leq t \leq T}e^{-\sigma (t - t_0)}\|\varphi(t)\|<\infty \}, \;\sigma\in (\alpha+KL,\beta-KL), 
\end{equation}
and $\mathcal{T}_T:\mathcal{F}_{\sigma,t_0,T}\times X \times \mathbb{R} \rightarrow \mathcal{F}_{\sigma,T}$ be defined by
\begin{equation}
\label{equ:mathcalT t}
\mathcal{T}_T(\varphi_T, x, t_0)(t) = \Lambda_A(t,t_0)x + \int_{t_0}^t \Lambda_A(t,s)\;F(\varphi_T(s))\;ds - \int_t^{T} \Lambda_B(t,s)\;G(\varphi(s)_T)\;ds,\quad t_0\leq t \leq T.
\end{equation}
We summarize the basic properties and facts of $\mathcal{T}_T$.  Since the proof is similar to the one in \cite{AulbachWanner} and \cite{foliation}, we omit it.
\begin{prop}
	Let $\mathcal{T}_T$ be defined as in \eqref{equ:mathcalT t}, and $x\in X$ and $t_0 \in \mathbb{R}$ be given. 
	\begin{enumerate}
		\item $\mathcal{T}_T(x, t_0, \cdot)$ is a contraction mapping in $\mathcal{F}_{\sigma,T}$.
		\item $\mathcal{T}_T(x, t_0, \cdot)$ has a unique fixed point, denoted by $\varphi_T$.
		\item $\varphi_T$ satisfies \eqref{equ:stable bvp T}.
	\end{enumerate}
	
\end{prop}
At this point, we have stated the basic results in regards to the solution of \eqref{equ:stable bvp T}.  We are ready to give the convergence dependence $T$ result. 
\begin{thm}
	\label{thm:conv depend T}
	Let $x\in X$ and $t_0 \in \mathbb{R}$ be given, and $\varphi$ and $\varphi_T$ be the fixed point of $\mathcal{T}(x, t_0, \cdot)$ and $\mathcal{T}_T(x, t_0, \cdot)$, respectively.  Then
	\begin{equation}
		\| \varphi - \varphi_T \|_{\sigma, T} \leq \frac{e^{\kappa} \kappa}{1 - \kappa} \|x_0\| e^{(\alpha - \sigma)(T-t_0)},
	\end{equation}
	where $\kappa := \max\{ \frac{KL}{\sigma - \alpha}, \; \frac{KL}{\beta - \sigma} \}$.
\end{thm}
\begin{proof}
	For $t \in [0, T]$, take the difference between $\varphi$ and $\varphi_T$:
	\begin{align*}
		\varphi(t) - \varphi_T(t) = \int_{t_0}^t \Lambda_A(t,s)\;[ F(\varphi(s)) - F(\varphi_T(s)) ]\;ds &- \int_t^{T} \Lambda_B(t,s)\;[ G(\varphi(s)) - G(\varphi_T(s)) ]\;ds \\
		&-\int_T^{\infty} \Lambda_B(t,s)\; G(\varphi(s))\;ds.
	\end{align*}
	Take the norm and multiply $e^{-\sigma (t-t_0)}$ on the both sides to obtain
	\begin{align*}
		e^{-\sigma (t - t_0)} \| \varphi(t) - \varphi_T(t) \| \leq & KL \max\{ \int_{t_0}^t e^{(t-s)(\alpha-\sigma)} e^{-\sigma (s - t_0)} \| \varphi(s) - \varphi_T(s) \|  \;ds,\;  \int_t^{T} e^{(t-s)(\beta - \sigma)} e^{-\sigma (s-t_0)} \| \varphi(s) - \varphi_T(s) \| \;ds\} \\
		& + KL \int_T^{\infty} e^{(t-s)(\beta-\sigma)}\;e^{-\sigma (s-t_0)} \|\varphi(s)\|\;ds.
	\end{align*}
	By Lemma \ref{lem:tail of varphi}, we have
	\begin{equation}
		\label{equ:tail proof1}
		\sup_{t\geq T} e^{-\sigma (t-t_0)} \| \varphi(t) \| \leq e^{(\alpha - \sigma)(T-t_0)} \|x_0\| e^{\kappa}.
	\end{equation}
	Also, by the direct calculation,
	\begin{equation}
		\label{equ:tail proof2}
		KL\int_{T}^{\infty}e^{(t-s)(\beta - \sigma)}\;ds \leq \kappa.
	\end{equation}
	Therefore, combine \eqref{equ:tail proof1} and \eqref{equ:tail proof2} and one has
	\begin{align}
		KL \int_T^{\infty} e^{(t-s)(\beta-\sigma)}\;e^{-\sigma (s - t_0)} \|\varphi(s)\|\;ds &\leq  e^{(\alpha - \sigma)(T - t_0)} \|x_0\| e^{\kappa} KL \int_{T}^{\infty}e^{(t-s)(\beta - \sigma)}\;ds \\
		&\leq e^{(\alpha - \sigma)(T - t_0)} \|x_0\| e^{\kappa} \kappa.
	\end{align}
	Hence, 
	\begin{equation}
		\label{equ:varphi - varphiT}
		e^{-\sigma (t-t_0)} \| \varphi(t) - \varphi_T(t) \| \leq \kappa \| \varphi - \varphi_T \|_{\sigma, T} +  e^{\kappa} \kappa \|x_0\| e^{(\alpha - \sigma)(T-t_0)}.
	\end{equation}
	Since the right side of \eqref{equ:varphi - varphiT} is independent of $t$, take the maximum over $[0, T]$
	\begin{align}
	 	&\| \varphi - \varphi_T \|_{\sigma, T} \leq \kappa \| \varphi - \varphi_T \|_{\sigma, T} +  e^{\kappa} \kappa \|x_0\| e^{(\alpha - \sigma)(T-t_0)}, \\
	 	\implies&\| \varphi - \varphi_T \|_{\sigma, T} \leq \frac{e^{\kappa} \kappa}{1 - \kappa} \|x_0\| e^{(\alpha - \sigma)(T-t_0)}.
	\end{align}
	Since $\sigma \in (\alpha+KL,\beta-KL)$, we have $\alpha - \sigma < 0$; therefore, $e^{(\alpha - \sigma)(T-t_0)} \rightarrow 0$, as $T\rightarrow \infty$.  This completes the proof.
\end{proof}
\begin{lem}\label{lem:tail of varphi}
	Let $\varphi$ be the fixed point of $\mathcal{T}$ for given $x\in X$ and $t_0\in \mathbb{R}$.  Then the following estimate holds
	\begin{equation}
		\label{equ:lem estimate}
		\sup_{t\geq T} e^{-\sigma (t - t_0)} \| \varphi(t) \| \leq e^{(\alpha - \sigma)(T - t_0)} \|x_0\| e^{\kappa}.
	\end{equation}
\end{lem}
\begin{proof}
	Since $\varphi$ is the fixed point of $\mathcal{T}$, for any $t\geq t_0$,
	\begin{align*}
		&\varphi(t) = \Lambda_A(t,t_0)x_0 + \int_{t_0}^t \Lambda_A(t,s)\;F(\varphi(s))\;ds - \int_t^{\infty} \Lambda_B(t,s)\;G(\varphi(s))\;ds\\
		\implies& e^{-\sigma (t-t_0)} \| \varphi(t) \| \leq e^{(\alpha-\sigma)(t-t_0)}\|x_0\| + KL \max\bigg\{ \int_{t_0}^t e^{(t-s)(\alpha-\sigma)} e^{-\sigma (s - t_0)} \| \varphi(s) \|  \;ds,\;  \\ &\int_t^{\infty} e^{(t-s)(\beta - \sigma)} e^{-\sigma (s - t_0)} \| \varphi(s) \| \;ds\bigg\}.
	\end{align*}
	By the Gronwall inequality, we have
	\begin{equation}
		e^{-\sigma (t-t_0)} \| \varphi(t) \| \leq  e^{(\alpha-\sigma)(t-t_0)}\|x_0\| e^{\kappa},
	\end{equation}
	and hence, the \eqref{equ:lem estimate} follows.
\end{proof}
At this point, we have shown that as $T\rightarrow\infty$, the solution of \eqref{equ:stable bvp T} converges to the solution of \eqref{equ:stable bvp}.  We conclude this section by showing how Theorem \ref{thm:conv depend T} can be used in practice.  There are two main sources of errors in the computation---numerical methods and convergence dependence of $T$.  Note that they are independent to each other.  Thus, if $T$ is too small, $\varphi_T$ is far from the $\varphi$, meaning the error induced by the convergence dependence of $T$ dominates the error by numerical methods.  For the most of time, the analytic manifolds are unavailable; thus, one could hardly know whether such $T$ is large enough, i.e. whether the error is dominated by the truncation error.  Thanks to Theorem \ref{thm:conv depend T}, we are able to give a crude lower bound of $T$ so that such truncation error is negligible.  Such error bound can be found \eqref{equ:varphi - varphiT}.  Let {\tt tol} be the tolerance, which can be machine zero or the order of numerical methods.  We need
\begin{equation}
	\frac{e^{\kappa} \kappa}{1 - \kappa} \|x_0\| e^{(\alpha - \sigma)(T-t_0)} \leq {\tt tol}.
\end{equation}
Since $\sigma$ is any number between $(\alpha+KL, \beta-KL)$, we choose $\sigma = (\alpha + \beta )/2$.  Therefore, by the fact that $\alpha - \beta < 0$, 
\begin{align*}
		& C \|x_0\| e^{\frac{\alpha-\beta}{2} (T-t_0)} \leq {\tt tol} \\
\implies& \log(C\|x_0\|) + \frac{\alpha-\beta}{2}(T-t_0) \leq \log({\tt tol}) \\
\implies& T \geq t_0 +  \frac{2}{\alpha-\beta} \log(\frac{{\tt tol}}{C\|x_0\|}),
\end{align*}
where $C = \frac{e^{\kappa} \kappa}{1 - \kappa}$.

\section{Numerical Results}\label{numericalresults}

%

\subsection{Two-dimensional nonautonomous test problem}

Consider the two dimensional ODE 
\begin{equation}\label{2Dtest}
\left\{ \begin{array}{c} \dot{v}= w -xy\cos(t)+(-y+x^2-2y^2+\sigma\cos(t))\sin(t) \\
	\dot{w} = -v + xy\sin(t)+(-y+x^2-2y^2+\sigma \cos(t))\cos(t)
\end{array}, \right.
\end{equation}
where $x = v\cos(t)-w\sin(t)$ and $y=v\sin(t)+w\cos(t)$.  This example is a variation of \cite{AJRoberts}, where we make a time dependent rotating transformation.    For this problem the inertial manifold is given by the approximate equality
\begin{equation}\label{2Dim}
v\sin(t) + w\cos(t)-(v\cos(t)-w\sin(t))^2-0.5\sigma (\cos(t)+\sin(t)) \approx 0.
\end{equation}
In the original example \cite{AJRobterts}, an explicit second order approximation in Taylor's series of the manifold is derived, and in our example, because of the transformation, we obtain a similar second order approximation for our problem \eqref{2Dim} as a reference.  We point out here that our method is more accurate than second order approximation in Taylor's series, but because of lack of closed form of analytic manifolds, the ``error'' term we used in this example is the difference between the computed results and the \eqref{2Dtest}.
We use this example to illustrate several important characteristics of our algorithm.  Figure 1 shows how the convergence of Algorithm \ref{IMalg} depends on the value of $T$.  As $T$ increases, initially the error of $v$ and $w$ from satisfying the approximate inertial manifold equations decreases at an exponential rate.  At $T\approx 1.7$ the error reaches a minimum and for $T > 1.7$ after which the error increases and remains above this minimal value.\\

Figure 2 shows the results for an experiment with \ref{IMalg} where we fix a value of $T$ and then vary the initial condition $\hat{w}_1(-T)$ we use to compute $\hat{w}_1(t)$ to construct the orthogonal transformation $Q(t)$.  The results show that for different initial conditions $\hat{w}_1(-T)$ we obtain different approximate trajectories each approximate trajectories converges to a point on the inertial manifold at a different rate.\\

Figures 3 and 4 show the results for an experiment with \ref{IMalgts} where we plot an approximate trajectory found using our time-stepping algorithm \ref{IMalgts} using the explicit Euler method for the time-stepping.  For the experiment we time-step from the point on the inertial manifold found in the experiment from Figure 1 where the error is minimized if we use $T=1.7$ so that we can try and maintain the accuracy of our approximated trajectory to the inertial manifold.  We see that for the step-sizes $\Delta t = 1E-2$ and $\Delta t = 1E-3$ that our time-stepping method at first retains the initial accuracy of the first approximated point to the inertial manifold.  However after several time steps the error of the approximated trajectory from the inertial manifold loses becomes less accurate by a factor of $10$.
\begin{figure}[h!]
\begin{center}
\includegraphics[scale=0.55]{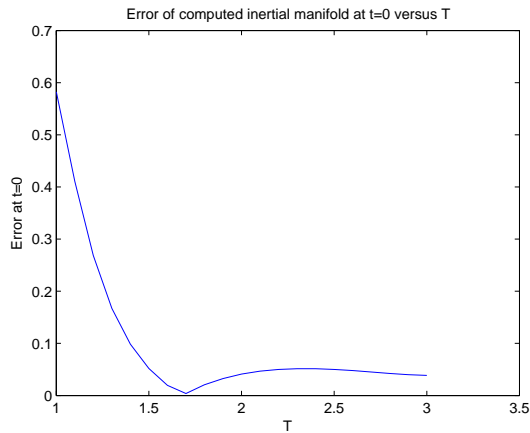}
\end{center}
\caption{Plot of the error of the computed $v$ and $w$ from satisfying the approximate inertial manifold equations at time $t=0$ versus $T$.   For each computation we used $\hat{w}_1(-T)=0.3$, $\sigma=0.1$, and $v(0)=1$ and error tolerances in the boundary value problem solver are set to $1E-3$.   The minimum error obtained is 3.863E-3 at $T=1.7$.}
\end{figure}

\begin{figure}[h!]
\begin{center}
\includegraphics[scale=0.55]{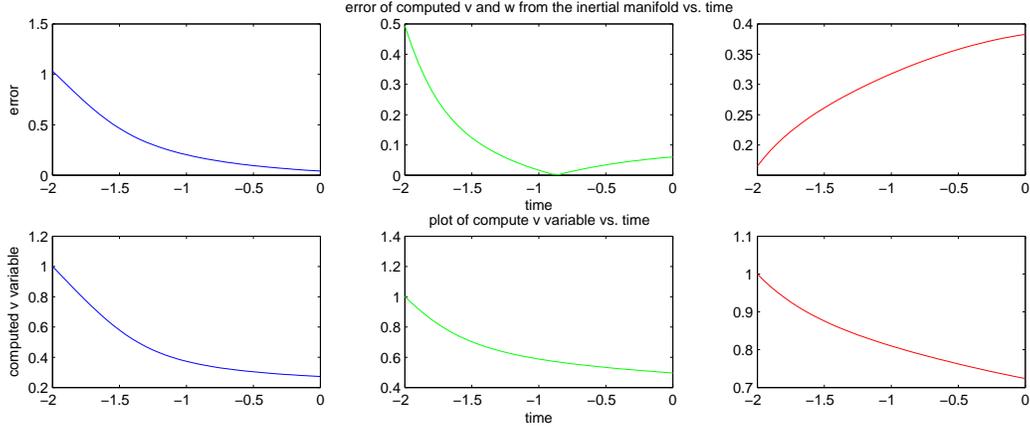}
\end{center}
\caption{First row: Pointwise plot of the error of the computed $v$ and $w$ from satisfying the left hand side of \eqref{2Dim} where from left to right we use initial conditions $\hat{w}_1(-T)=0.3$,$\hat{w}_1(-T)=0.6$, $\hat{w}_1(-T)=0.9$.  Second row:  Plot of the computed $v$ versus time for several values of $\hat{w}_1(-T)$ where from left to right we use initial conditions $\hat{w}_1(-T)=0.3$,$\hat{w}_1(-T)=0.6$, $\hat{w}_1(-T)=0.9$.  For each computation we used $T=2$ and $v(0)=1$ and and error tolerances in the boundary value problem solver are set to $1E-3$.}
\end{figure}

\begin{figure}[h!]
\includegraphics[scale=0.55]{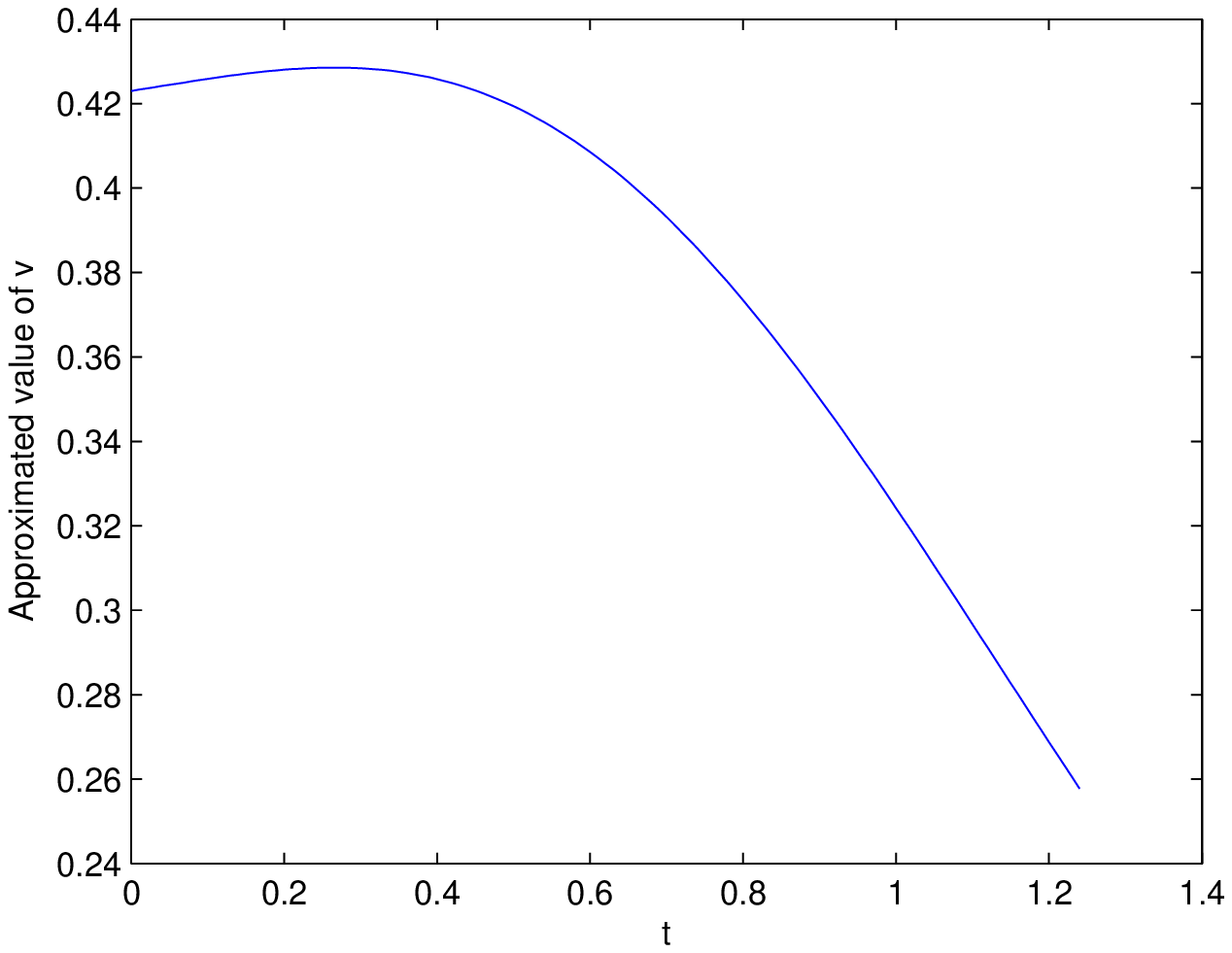}\includegraphics[scale=0.55]{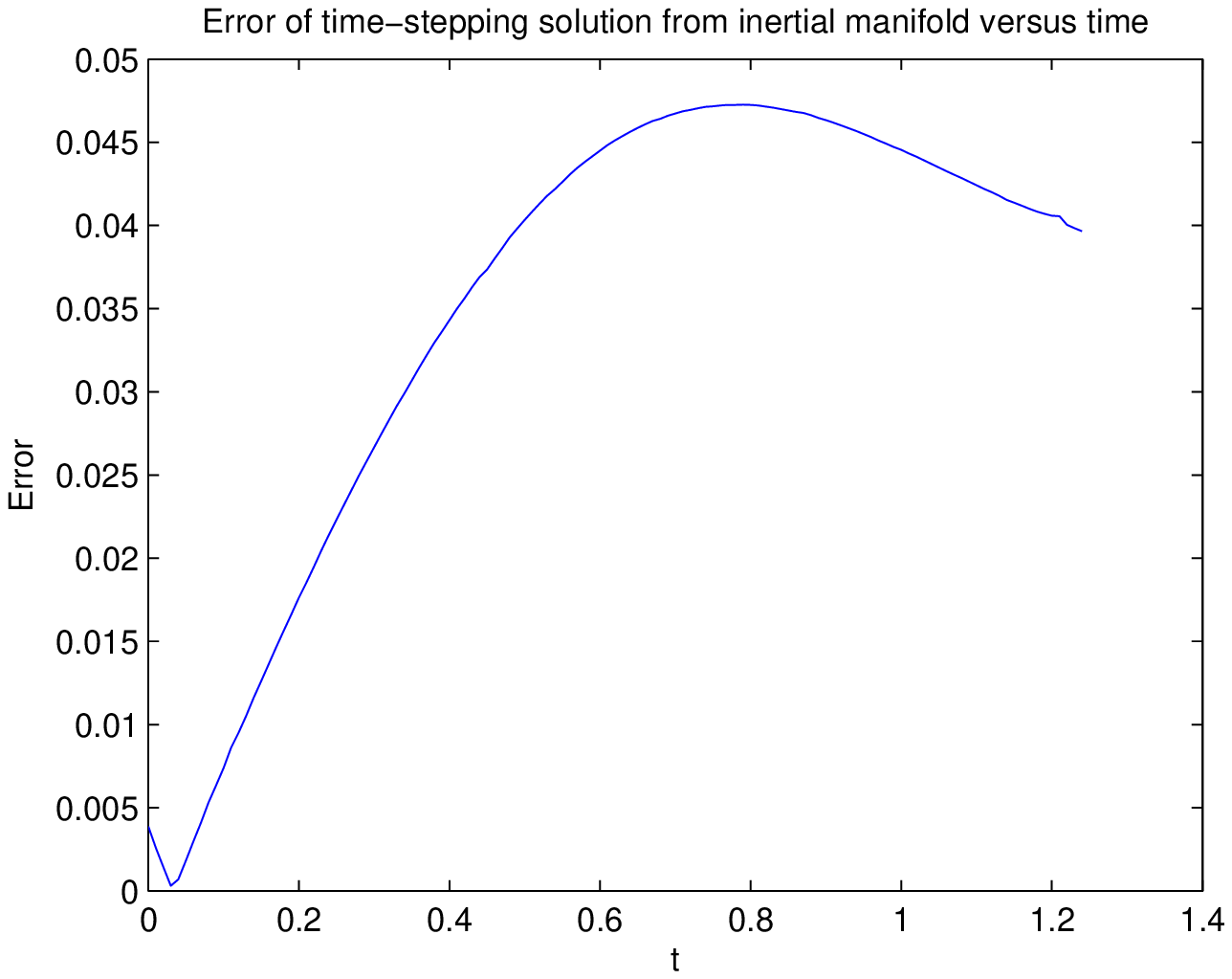}
\caption{Results of an experiment using time-stepping algorithm using explicit Euler method with step-size $1E-2$.  Left:  Plot of the computed $v$ versus. Right: Pointwise plot of the error of the computed $v$ and $w$ from satisfying the left hand side of \eqref{2Dim} for several values.   For the computation we used $T=1.7$ and $\hat{w}_i(-T)=0.3$.}
\end{figure}
\begin{figure}[h!]
\includegraphics[scale=0.55]{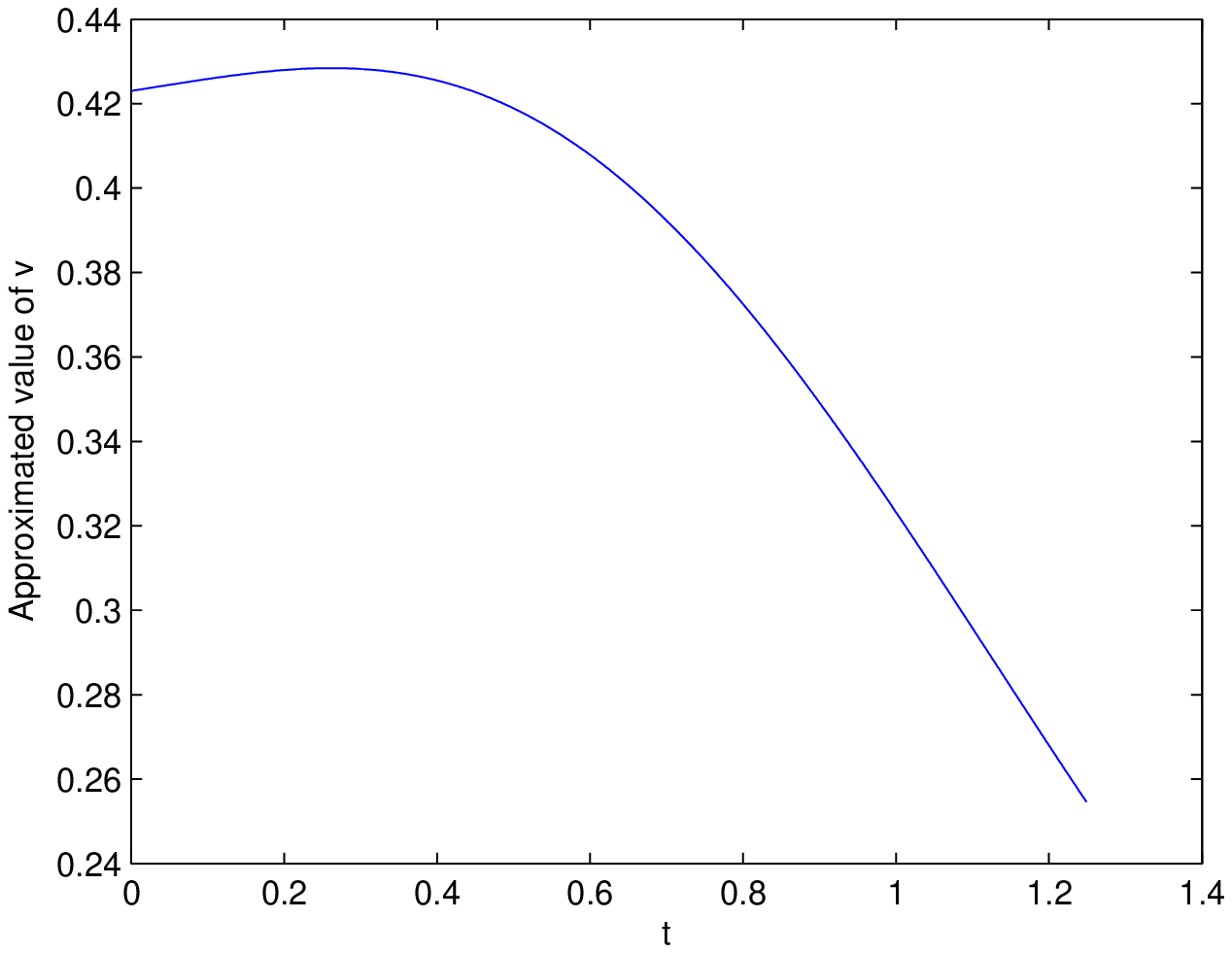}\includegraphics[scale=0.55]{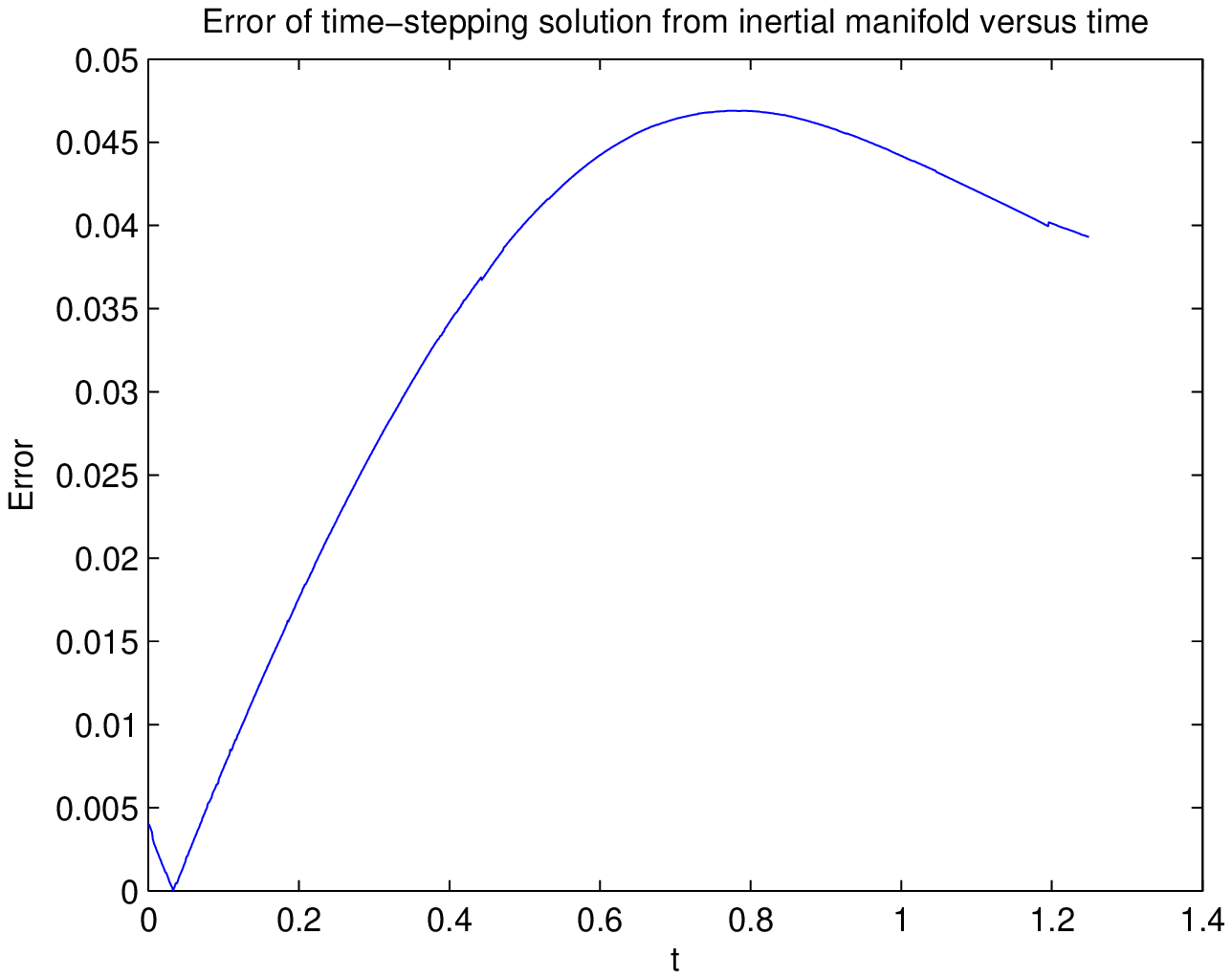}
\caption{Results of an experiment using time-stepping algorithm using explicit Euler method with step-size $1E-3$.  Left:  Plot of the computed $v$ versus. Right: Pointwise plot of the error of the computed $v$ and $w$ from satisfying the left hand side of \eqref{2Dim} for several values.   For the computation we used $T=1.7$ and $\hat{w}_i(-T)=0.3$.}
\end{figure}

\subsection{KSE Equation}

As a second example consider the Kuramoto-Sivashinsky equation
in the form
\begin{equation}\label{eq:ourKSE}
\pd{\ut}{\tau} + 4\pd{^4\ut}{y^4} +
\vartheta\Bigl[\pd{^2\ut}{y^2} + \ut\pd{\ut}{y} \Bigr]= 0\;,
\end{equation}
with $\ut(y,t)=\ut(y+2\pi,t)$, and $\ut(y,t)=-\ut(-y,t)$.
With the change of variables
$$
-2w(s,y)=\ut(\xi s / 4,y), \qquad  \xi=\frac{4}{\vartheta}\;.
$$
\eqref{eq:ourKSE} can be written as
\begin{equation}\label{eq:CCP}
w_s=(w^2)_y-w_{yy}-\xi w_{yyyy}\;.
\end{equation}
All computations were performed for $\xi=0.02991$,  $\vartheta=133.73454$,
one of the parameter values considered in \cite{DJRV,DJVV11}.  We consider a spectral discretization in Fourier space using a standard Galerkin truncation (see \cite{DJRV} for more details).  Figure 5 shows plots of an experiment where we use Algorithm \ref{IMalg} to approximate a point on the inertial manifold and then use the Matlab's ode45 initial value problem solver with point as its initial condition to see how close this point is to being on the inertial manifold for various values of $T$.  These plots illustrate that approximation of a point on inertial manifold found using Algorithm \ref{IMalg} initially improves as we increase $T$ and after a certain value this approximation fails to improve.

\begin{figure}[h!]
\begin{center}
\includegraphics[scale=0.5]{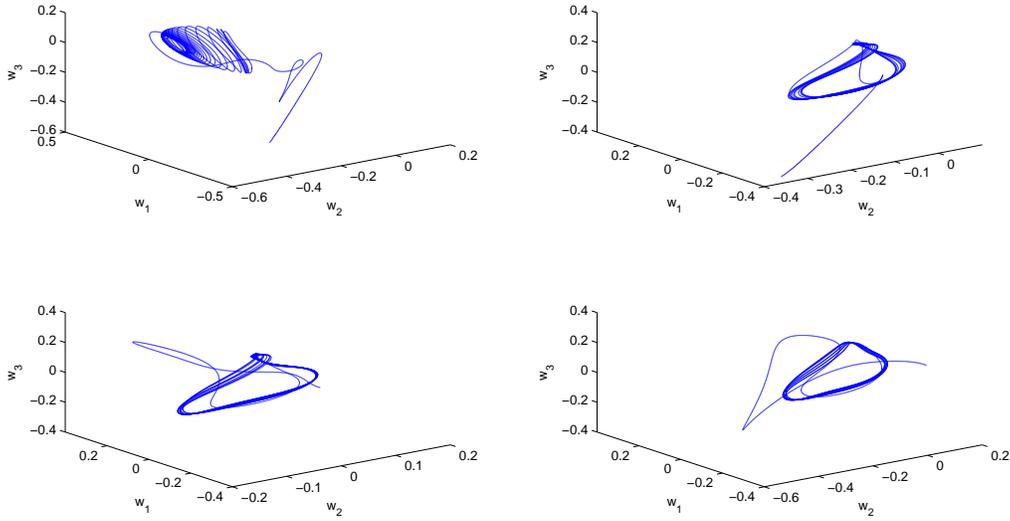}
\end{center}
\caption{Plot of the first three coordinates of the solution on $[0,10]$ using the output of Algorithm \ref{IMalg} as the initial condition using various values of $T$.  Top left is $T = 0.001$, top right is $T = 0.01$, bottom left is $T=.08$, and bottom right is $T=0.1$.   For each plot we use $p=6$, error tolerances of $1E-3$ in the boundary value problem solver and $1E-4$ in the initial value problem solver.  The initial boundary conditions used were $y(0) =\frac{1}{\sqrt{6}}(-1,1,-1,1,-1,1)^T$, $\hat{w}_i(-T)=0$ for $i=1,\hdots,p$, and $x(-T)=(0,\hdots,0)^T \in \mathbb{R}^9$.  The results displayed are plots of the first three coordinates of the solution . }
\end{figure}

\subsection{Two layer Lorenz}

Consider the following 2 layer Lorenz system that appears in \cite{Lor96,FV-E04} given by the equations
\begin{equation}\label{eq:2LL}
\left\{\begin{array}{c}
\dot{x}_k = x_{k-1}(x_{k+1}-x_{k-2})-x_k + F + z_k, \quad k=1,\hdots, K \\
\dot{y}_{j,k} = \epsilon^{-1} \left(y_{j+1,k}(y_{j-1,k}-y_{j+2,k})-y_{j,k} + h_y x_k \right), \quad j=1,\hdots, J
\end{array}
\right.
\end{equation}
where we take $K=5$, $J=4$, $\epsilon=0.5$, $h_x = -1$, $h_y = 1$, and $z_k = \frac{h_x}{J}\sum_{j}y_{j,k}$.  The approximate real parts of the eigenvalues of the linearization of this system about the origin lie in the interval $[-32,-1]$.  We find that there are $8$ eigenvalues with real part less than $-15$ and $17$ eigenvalues with real part greater than approximately $-8$.  This suggests a splitting of the system with $p=8$, as opposed to splitting the system with $p=5$ into $x_k$ and $y_{j,k}$ as suggested by the form of the equations.  In Figure 6 we display the results of an experiment where we use Algorithm \ref{IMalg} to approximate a point on the inertial manifold and then use the Matlab's ode45 initial value problem solver with point as its initial condition to see how close this point is to being on the inertial manifold for various values of $T$ and $p$.

\begin{figure}[h!]
\begin{center}
\includegraphics[scale=0.6]{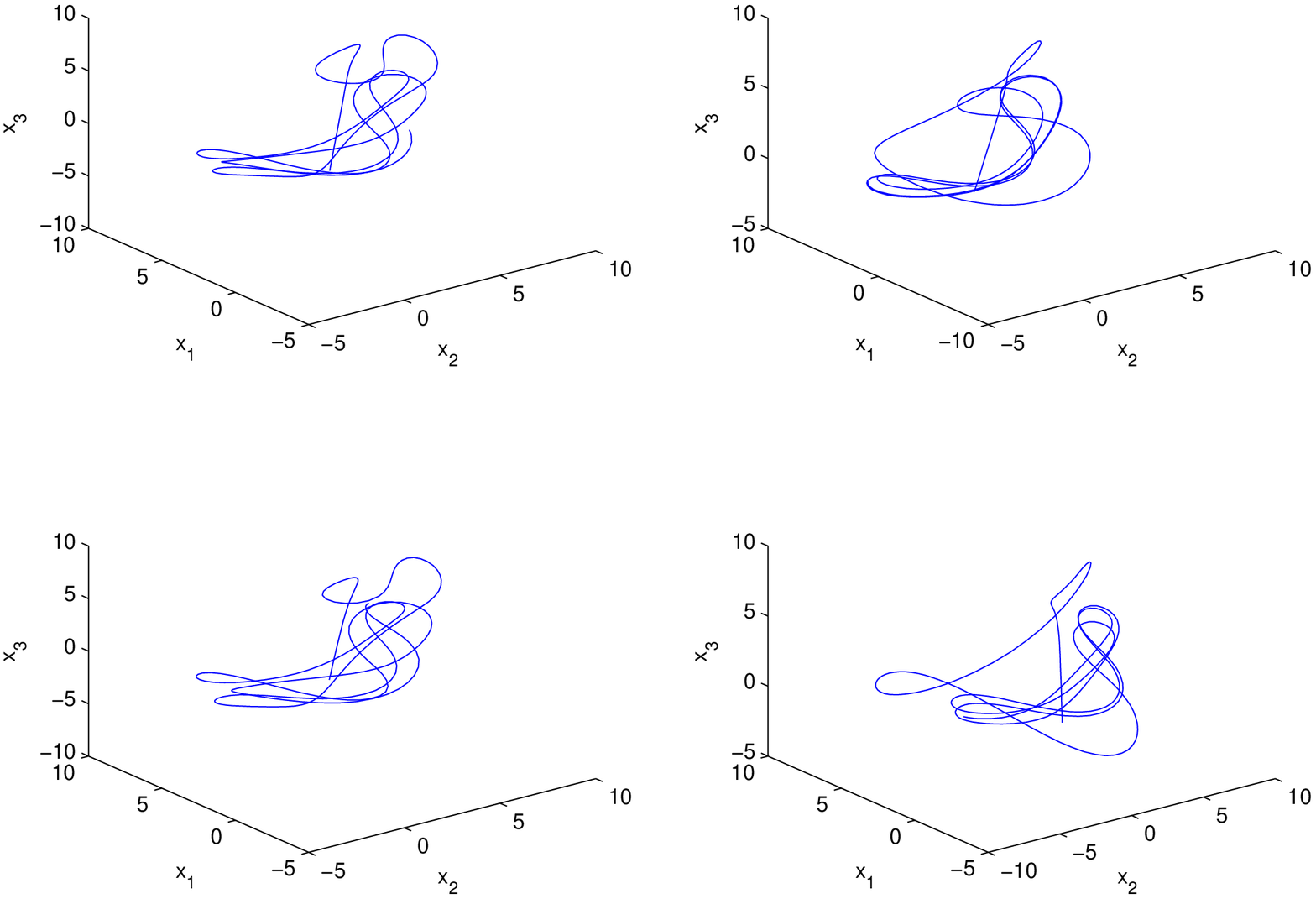}
\end{center}
\caption{Plot of the first three coordinates of the solution on $[0,6]$ using the output of Algorithm \ref{IMalg} as the initial condition using various values of $T$ and $p$.  Top left is $T = 0.001$ and $p=8$, top right is $T = 0.01$ and $p=8$, bottom left is $T=.01$ and $p=5$, and bottom right is $T=0.1$ and $p=12$.   For each plot we use error tolerances of $1E-3$ in the boundary value problem solver and $1E-4$ in the initial value problem solver.  The initial boundary conditions used were $y(0) =\frac{1}{\sqrt{6}}(-1,1,-1,1,-1,1)^T$, $\hat{w}_i(-T)=0$ for $i=1,\hdots,p$, and $x(-T)=(0,\hdots,0)^T \in \mathbb{R}^9$.  The results displayed are plots of the first three coordinates of the solution. }
\end{figure}

\section{Discussion}\label{Discuss}

In this paper we have developed a boundary value formulation for the computation of trajectories on inertial manifolds of nonlinear systems that satisfy a gap condition.  A point on the inertial manifold is found as the solution of a boundary value problem with a boundary term at $-\infty$.  A trajectory on the inertial manifold is approximated by using a standard initial value problem solver that evaluates the right-hand side of the differential equation by calling the boundary value problem solver.  We have applied our method to a variety of challenging problems to demonstrate that our algorithm is useful for the computation of trajectories on the inertial manifolds.  However, our method suffers from the drawback of needing to properly select a value of $T \approx \infty$ that is neither too large nor too small and needing to select boundary conditions that allows us to efficiently construct the decoupling transformation $Q(t)$.  

There is still much work to be done on the analysis and development of our algorithms.  We would like to find better bounds on the value of $T \approx \infty$ so that we can minimize or control the error in computing points on the inertial manifold without access to any exact equations describing the manifold or the exact solution of the differential equation.  Additionally we would like to investigate the interaction of the local error of the time-stepping method that solves the initial value problem with the error of the method that approximate points on the inertial manifold so that we can use $T$ and the step-size $\Delta t$ of the time-stepping method to control the error of the approximate trajectory on the manifold.

\bibliographystyle{plain}
\bibliography{decoup}

\end{document}